\documentclass[12pt]{amsart}
\usepackage{amssymb}
\usepackage{amsfonts}
\usepackage{latexsym}
\usepackage{amscd}
\usepackage[mathscr]{euscript}
\usepackage{enumitem}
\usepackage{xy} \xyoption{all}

\newfont{\gd}{eufm10 scaled \magstep1}
\newfont{\gs}{eufm7 scaled \magstep1}
\newfont{\gss}{eufm5 scaled \magstep1}

\newcommand{\N}{{\mathbb{N}}}
\newcommand{\Z}{{\mathbb{Z}}}

\newcommand{\ndr}{\mbox{\rm Endr}}

\newcommand{\sop}{S^{\text{\rm op}}}
\newcommand{\sinvt}{S^{-1}T}

\newcommand{\sinva}{S^{-1}A}

\newcommand{\ndz}{\mbox{\rm End}_{\mathbb{Z}}}
\newcommand{\sopat}{\sop*_\alpha A*_\alpha T}
\newcommand{\sopas}{\sop*_\alpha A*_\alpha S}
\newcommand{\SopaS}{\sop*_{\overline{\alpha}} (S_-AS_+)*_{\overline{\alpha}}S}
\newcommand{\Zplus}{\Z^+}

\newcommand{\aut}{\mbox{\rm Aut}}
\newcommand{\alsom}{{\widehat {\alpha }}}
\newcommand{\sopasprime}{\sop*_{\alpha'}(eAe)*_{\alpha'}S}

\newtheorem{lemma}{Lemma}[section]
\newtheorem{corollary}[lemma]{Corollary}
\newtheorem{theorem}[lemma]{Theorem}
\newtheorem{proposition}[lemma]{Proposition}
\newtheorem{remark}[lemma]{Remark}
\newtheorem{definition}[lemma]{Definition}
\newtheorem{example}[lemma]{Example}

\newtheorem{noname}[lemma]{}

\begin{document}

\title[Dilations and full corners]{Dilations and full corners on fractional skew monoid rings}%
\author{E. Pardo}
\address{Departamento de Matem\'aticas, Facultad de Ciencias\\ Universidad de C\'adiz, Campus de
Puerto Real\\ 11510 Puerto Real (C\'adiz)\\ Spain.}
\email{enrique.pardo@uca.es}\urladdr{https://sites.google.com/a/gm.uca.es/enrique-pardo-s-home-page/} \urladdr{http://www.uca.es/dpto/C101/pags-personales/enrique.pardo}

\thanks{The author supported by PAI III grants FQM-298 and P07-FQM-7156 of the Junta de Andaluc\'{\i}a, by the DGI-MICINN and European Regional Development Fund, jointly, through Project MTM2011-28992-C02-02 and by 2009 SGR 1389 grant of the Comissionat per Universitats i Recerca de la Generalitat de Catalunya.}




\begin{abstract}
In this note we will show that the dilation result obtained for fractional skew monoid rings, in the case of a cancellative left Ore monoid $S$ acting on a unital ring $A$ by corner isomorphisms, holds in full generality. We apply this result to the context of semigroup $C^*$-crossed products.
\end{abstract}
\maketitle

\section*{Introduction}

In his pioneering paper \cite{C1}, Cuntz defined the algebras $\mathcal{O}_n$ and presented them as crossed products by endomorphisms. Inspired by this construction, Paschke \cite{Paschke} gave a construction of a $C^*$-algebraic crossed product $A\rtimes _{\alpha}\N$ associated to a not necessarily unital $C^*$-algebra endomorphism $\alpha$ on a $C^*$-algebra $A$. Later, R\o rdam \cite{rordamclassif} used Paschke's construction, together with the Pimsner-Voiculescu exact sequence associated to an automorphism \cite[Theorem 10.2.1]{Black}, to realize any pair of countable abelian groups $(G_0,G_1)$ as $(K_0(B),K_1(B))$ for a certain purely infinite, simple, nuclear separable $C^*$-algebra $B$. Paschke's $C^*$-algebraic construction has been generalized to other semigroups, see e.g. \cite{LacaRae, Larsenirish, Larsencanad,LarsenRae, Murphy1, Murphy2}. 

A particularly interesting tool, first remarked by Cuntz, then used by R\o rdam, and lately extended by Paschke in the case of the actions of nonnegative integers, was developed by Laca \cite{Laca} in full generality. Laca proved that, in the case of a cancellative left Ore monoid $S$ with enveloping group $G$, there is an isomorphism between the semigroup $C^*$-crossed product $A\times _{\alpha} S$ and a full corner of the group crossed product $A_{S}\times _{\widehat{\alpha}} G$ (where $A_{S}$ is a direct limit  associated to $A$ and $\alpha$), whenever the maps induced by the action are injective (with hereditary image); this is the so-called dilation of $A\times _{\alpha} S$, a construction which relies upon previous work of Murphy \cite{Murphy2}. As a consequence, many properties can be faithfully transfered from the semigroup $C^*$-crossed product to a Morita equivalent group $C^*$-crossed product. Thus, the study of the properties of the new construction will benefit of the full developed theory for the second kind of $C^*$-algebras. \vspace{.2truecm}

In \cite{AGGBP}, Ara, Gonz\'alez-Barroso, Goodearl and the author developed a purely algebraic analog of Paschke's construction with respect to monoid actions on rings: for a monoid $T$ acting on a unital ring $A$ by endomorphisms and a submonoid $S$ of $T$ satisfying the left denominator conditions, it is constructed a fractional skew monoid ring $\sopat$ which satisfies a universal property analogous that of skew group rings. In the case of $S=T=\Zplus$, the similarity of the fractional skew monoid ring with a skew-Laurent polynomial ring lets adapt the Bass-Heller-Swan-Farrell-Hsiang-Siebenmann Theorem to these rings for computing their $K_n$ groups ($n\in \Z$) \cite{AB, ABC}. As an application, this result allows to compute $K$-Theory for Leavitt path algebras \cite{AA1, AMFP}, the algebraic counterpart of graph $C^*$-algebras \cite{Raeburn}. Moreover, it is shown that an analog of Laca's dilation construction holds for fractional skew monoid rings, when the action restricts to corner isomorphisms (the algebraic analog of Laca's requirements). \vspace{.2truecm} 

The origin of the present work relies in the problem of extending the dilation construction to actions enjoying less restrictive properties. More concretely, we try to answer a concrete question posed to the author by Joachim Cuntz: 
\begin{quotation}
\textbf{Is it possible to extend the dilation construction of \cite{AGGBP} in the case of actions by \underline{unital} endomorphisms?} 
\end{quotation}
In this paper we show that the question has an affirmative answer: whenever the semigroup is cancellative, the dilation result of \cite{AGGBP} and \cite{Laca} hold with no restriction about hereditariness of the images. \vspace{.2truecm}

The contents of this paper can be summarized as follows. In Section 1, we recall the construction of a fractional skew monoid ring and some basic properties. Also, we recall the construction of the semigroup $C^*$-algebra crossed product in Laca's sense \cite{Laca}, and we analyze the fractional skew monoid ring construction from the point of view of covariant pairs. In Section 2 we show that, after changing the basis ring, it is possible to assume that the action is given by corner isomorphisms. In Section 3, we recall the dilation results of \cite[Section 3]{AGGBP}, and we prove the announced result and some consequences. In Section 4, we explain how to transfer these results to the context of $C^*$-algebras, then recovering and extending Laca's results. Finally, in Section 5,  we analyze the scope of application of the results of Section 3, by extending them to the case of not necessarily cancellative left denominator monoids. 

\section{Basic elements}

In this section we will recall the algebraic and analytic version of skew semigroup algebras, and we will look at the connections between both constructions.

\subsection{Fractional skew monoid rings}

We recall the construction of a fractional skew monoid ring, as well as the basic results we will need. The definitions and properties are borrowed from \cite[Section 1]{AGGBP}.

\begin{noname}\label{data1} {\rm
We begin by fixing the basic data needed for the construction. Let $A$ be a unital ring, and $\ndr(A)$
the monoid of not necessarily unital ring endomorphisms of $A$.

Let $T$ be a (multiplicative) monoid and $\alpha: T\rightarrow \ndr(A)$ a monoid
homomorphism, written $t\mapsto \alpha_t$. For $t\in T$, set $p_t= \alpha_t(1)$,
an idempotent in $A$. Then $\alpha_t$ can be viewed as a unital
ring homomorphism from $A$ to the corner $p_tAp_t$. For $s,t\in
T$, we have $p_{st}= \alpha_{st}(1)= \alpha_s \alpha_t(1)=
\alpha_s(p_t)$.

Let $S\subseteq T$ be a submonoid satisfying the left denominator
conditions, i.e., the left Ore condition and the monoid version of
left reversibility: whenever $t,u\in T$ with $ts=us$ for some
$s\in S$, there exists $s'\in S$ such that $s't=s'u$. Then there
exists a monoid of fractions, $S^{-1}T$, with the usual properties
(e.g., see \cite[Section 1.10]{ClPr} or \cite[Section 0.8]{Cohn}). Notice that, even in the case that $S=T$, the monoid $S$ does not need to be cancellative (e.g. any inverse monoid \cite[Example 1.5(2)]{pic}). But if $S$ is cancellative, then the left Ore condition implies left reversibility.}
\end{noname}

\begin{definition}[{\cite[Definition 1.2]{AGGBP}}]\label{newring} {\rm We denote by $\sopat$
a unital ring $R$ equipped with a unital ring
homomorphism $\phi: A\rightarrow R$ and monoid homomorphisms
$s\mapsto s_-$ from $\sop\rightarrow R$ and $t\mapsto t_+$ from
$T\rightarrow R$, universal with respect to the following
relations:

\begin{enumerate}

\item $t_+\phi(a)= \phi\alpha_t(a)t_+$ for all $a\in A$ and $t\in
T$;

\item $\phi(a)s_-= s_-\phi\alpha_s(a)$ for all $a\in A$ and $s\in
S$;

\item $s_-s_+= 1$ for all $s\in S$;

\item $s_+s_-= \phi(p_s)$ for all $s\in S$.

\end{enumerate}}
\end{definition}

The existence of such a ring follows from classical arguments. The construction above also applies when $A$ is an algebra over a field $K$ or a $\ast$-algebra, and the ring endomorphisms $\alpha_t$ for $t\in
T$ are $K$-linear or $\ast$-homomorphisms. Because of Property (3), Property (2) can be replaced by 
\begin{quotation}
(2') $s_+\phi(a)s_{-}= \phi\alpha_s(a)$ for all $a\in A$ and $s\in S$,
\end{quotation}
and Property (4) becomes redundant. Moreover, if $s_{+}$ is an isometry for all $s\in S$ (i.e. if $s_-=(s_+)^*$) then Property (3) becomes redundant too. We have the following fact:

\begin{proposition}[{\cite[Corollary 1.5 \& Proposition 1.6]{AGGBP}}]\label{sum} $\mbox{ }$
\begin{enumerate}
\item $R= \sum_{s\in S,\, t\in T} s_-\phi(A)t_+=
\sum_{s\in S,\, t\in T} s_-\phi(p_sAp_t)t_+$. 
\item The ring $R$ has an $\sinvt$-grading $R=
\bigoplus_{x\in\sinvt} R_x$ where each $R_x= \bigcup_{s^{-1}t=x}
s_-\phi(A)t_+$.
\end{enumerate}
\end{proposition}

Now, assume that $S$ is \emph{left saturated} in $T$:
whenever $s\in S$ and $t\in T$ such that $ts\in S$, we must have
$t\in S$; when $S=T$, this hypothesis is clearly fulfilled. Under this additional hypothesis, we can show the following result:

\begin{proposition}[{\cite[Corollary 1.11]{AGGBP}}]\label{kernel}$\mbox{ }$
\begin{enumerate}
\item Let $s\in S$, $t\in T$, and $a\in A$. Then
$s_-\phi(a)t_+=0$ if and only if $p_sap_t\in \ker(\alpha_{s'})$
for some $s'\in S$. In particular, $\ker(\phi)= \bigcup_{s'\in S}
\ker(\alpha_{s'})$.
\item The ideal $I=\ker(\phi)$ satisfies $\alpha_s^{-1}(I)= I$ for all
$s\in S$ and
$\alpha_t(I) \subseteq I$ for all $t\in T$.
\item  $\alpha$ induces a monoid homomorphism $\alpha':
T\rightarrow \ndz(A/I)$, and $\alpha'_s$ is injective for all
$s\in S$.
\item $\sop*_\alpha A*_\alpha T= \sop*_{\alpha'}
(A/I)*_{\alpha'} T$.
\end{enumerate}
\end{proposition}

As Proposition \ref{kernel} shows, we can reduce the construction to the situation where $\alpha_s$ is injective for all $s\in S$. In this case, $\phi$ is injective by Proposition \ref{kernel}(1), and so we can identify $A$ with the unital subring $\phi(A)$ of $R$. 
 
\subsection{Semigroup $C^*$-crossed products}
 
We recall the definition of a semigroup $C^*$-crossed product. The definitions and properties are borrowed from \cite[Subsection 1.3]{Laca}.
 
\begin{noname}\label{data1bis}{\rm
Let $A$ be a unital $C^*$-algebra, let $\alpha$ be an action of a discrete semigroup $S$ by not necessarily unital endomorphisms of $A$, and let $H$ be a complex Hilbert space. Then, a covariant representation of the (semigroup) dynamical system $(A,S, \alpha)$ is a pair $(\pi, V)$ in which
\begin{enumerate}
\item[(i)] $\pi$ is a unital representation of $A$.
\item[(ii)] $V:S\rightarrow \text{Isom}(H)$ is an isometric representation of $S$, that is, $V_sV_t=V_{st}$ for every $s,t\in S$.
\item[(iii)] The covariance condition $\pi (\alpha_t(a))=V_t\pi (a)V_t^*$ holds for every $a\in A$ and every $t\in S$.
\end{enumerate}
}
\end{noname}
 
\begin{definition}\label{newalgebra}{\rm
Given a dynamical system $(A,S,\alpha)$ as above, the semigroup $C^*$-crossed product associated to it is a $C^*$-algebra $A\times_{\alpha} S$ together with a unital homomorphism $i_A: A\rightarrow A\times_{\alpha} S$ and a representation of $S$ as isometries $i_S: S\rightarrow A\times_{\alpha} S$ such that
\begin{enumerate}
\item $(i_A, i_S)$ is a covariant representation for $(A,S, \alpha)$.
\item For any other covariant representation $(\pi, V)$ there is a representation  $\pi\times V$ of $A\times_{\alpha} S$ such that $\pi=(\pi \times V)\circ i_A$ and $V=(\pi\times V)\circ i_S$.
\item $A\times_{\alpha} S$ is generated by $i_A(A)$ and $i_S(S)$ as a $C^*$-algebra.
\end{enumerate}
}
\end{definition}

The existence of a nontrivial universal object associated to $(A,S, \alpha)$ depends on the existence of a nontrivial covariant representation. For general endomorphisms such representation does not need to exist. But if the endomorphisms are injective, then the nontriviallity of $A\times_{\alpha} S$ will follow from its realization as a corner in a nontrivial classical $C^*$-crossed product, which turns out to be nontrivial by an argument similar to the one used in \cite[Proposition 2.2]{Stacey} (see \cite[Remark 2.5]{Laca}). The best result in this direction was proved by Laca when $S$ is a cancellative Ore semigroup acting by injective endomorphisms of $A$ \cite[Theorem 2.1 \& Theorem 2.4]{Laca}. \vspace{.2truecm}

\subsection{The algebraic construction from the analytic point of view}
 
We can understand fractional skew monoid rings in terms of algebraic covariant pairs. For, we will follow the same scheme used in Subsection 1.2.

\begin{noname}\label{eslapera}{\rm 
Let $A$ be a unital $K$-algebra, let $T$ be a monoid, let $S\subseteq T$ be a submonoid satisfying the left denominator conditions, and let $\alpha$ be an action of $T$ by not necessarily unital endomorphisms of $A$. Then, given a (infinite dimensional) $K$-vector space $H$, a covariant representation of the algebraic dynamical system $(A, T, S, \alpha)$ is a pair $(\phi, V)$ in which
\begin{enumerate}
\item $\phi:A \rightarrow \text{End}_K(H)$ is a unital homomorphism.
\item $V:T\rightarrow  \text{End}_K(H)$ is a monoid homomorphism that restricts to an isometric representation $V_{\vert S}:S\rightarrow \text{Isom}(H)\subset \text{End}_K(H)$ of $S$ , that is, $V_sV_t=V_{st}$ for every $s,t\in T$, and $V_s$ is an ``adjoinable'' endomorphism such that $V_s^* V_s=\text{Id}_H$ for every $s\in S$.
\item $\pi (\alpha_t(a))V_t=V_t\pi (a)$ holds for every $a\in A$ and every $t\in T$.
\item The covariance condition $\pi (\alpha_t(a))=V_t\pi (a)V_t^*$ holds for every $a\in A$ and every $t\in S$.
\end{enumerate}
}
\end{noname}

By Definition \ref{newring}, $\sopat$ turns out to be a universal initial object for the category of covariant representations of the algebraic dynamical system $(A,T,S, \alpha)$, in analogy with Definition \ref{newalgebra}. To be concrete, we obtain the equivalent version of Definition \ref{newring}. 

\begin{definition}\label{newcovariance}{\rm
Given an algebraic dynamical system $(A,T,S,\alpha)$ as above, the fractional skew monoid ring associated to it is a $K$-algebra $\sopat$ together with a unital homomorphism $\phi_A: A\rightarrow \sopat$ and a representation $\phi_T: T\rightarrow \sopat$ of $T$ restricting to a representation of $S$ as isometries $\phi_S: S\rightarrow \sopat$ such that:
\begin{enumerate}
\item $(\phi_A, \phi_T)$ is a covariant representation for $(A,T, S, \alpha)$.
\item For any other covariant representation $(\pi, \tau)$ there is a representation  $\pi\times \tau$ of $\sopat$ such that $\pi=(\pi \times \tau)\circ \phi_A$ and $\tau=(\pi\times \tau)\circ \phi_T$.
\item $\sopat$ is generated by $\phi_A(A)$ and $\phi_S(S)$ as a $K$-algebra.
\end{enumerate}
}
\end{definition}

We can separate two extreme cases:\vspace{.2truecm}

(A) If $S=\{1_T\}\subset T$, then, $\sopat =A\ast_{\alpha} T$ is the classical skew semigroup ring, with unique relation $t_+\phi (a)=\phi \alpha_t(a)$ for every $a\in A$ and every $t\in T$. This picture corresponds to Murphy's definition of the crossed product of $C^*$-algebras by endomorphisms given in \cite{Murphy1}. But since no extra restriction on the representation $V:T\rightarrow  \text{End}_K(H)$ is required, $A\ast_{\alpha} T$ cannot be completed to a $C^*$-algebra.\vspace{.2truecm}

(B) If $S=T$, then Property (2) of (\ref{eslapera}) corresponds to Property (ii) in (\ref{data1bis}). Also, Properties (3) and (4) of (\ref{eslapera}) become equivalent, and correspond to Property (iii) of (\ref{data1bis}). So, the definition of covariant representation in (\ref{eslapera}) recovers (\ref{data1bis}), and thus $\sop*_\alpha A*_\alpha S$ satisfies Definition \ref{newcovariance} with respect to (\ref{eslapera}).\vspace{.2truecm}

Under this point of view, Proposition \ref{kernel} says that given any algebraic dynamical system $(A,T,S, \alpha)$, we can construct a new algebraic dynamical system $(A/I, T,S, \alpha')$ such that:
\begin{enumerate}
\item For the universal covariant pair $(\phi_{A/I}, \phi_T)$, the map $\phi_{A/I}$ and the endomorphisms $\alpha'_s$ (for every $s\in S$) are injective.
\item Both algebraic dynamical systems have the same universal initial object.
\end{enumerate}
This means that, at the algebraic level, the injectivity of the endomorphisms is not a necessary requirement to have control of the nontriviality of $\sopat$.\vspace{.2truecm}

In the sequel, we will come back to this analytic picture, in order to understand how the results below extends Laca's achievements. 
 
\section{From injective morphisms to corner isomorphisms}\label{injtocorner}

In this section we will show that, in the construction of $\sopas$, we can always assume that the action of $S$ on $A$ is given by corner isomorphisms. Recall that given a ring $A$ and a nonzero idempotent $p\in A$, a corner isomorphism is a ring isomorphism $f:A\rightarrow pAp$. For example, in $\sopat$, if $s\in S$, $\alpha_s$ is injective and $\alpha_s(A)=p_sAp_s$, then $\alpha_s$ is a corner isomorphism. But even in the case of $\alpha_s$ being injective, it is not necessarily a corner isomorphism.

Let us fix the standing hypotheses, that we will assume as general as possible.

\begin{noname}\label{data5} {\rm
Let $A$ be a unital ring, let $S$ be a left Ore monoid satisfying left reversibility, and suppose that $\alpha :S\rightarrow \ndr(A)$ is an action of $S$ on $A$ by injective homomorphisms. Notice that we are assuming that it must exists at least one $s\in S$ such that $\alpha_s(A)\subsetneq p_sAp_s$.}
\end{noname}

\begin{lemma}\label{Lemma1nou}
Let $s,t\in S$. If $\widehat{s}, \widehat{t}\in S$ and $\widehat{s}t=\widehat{t}s$, then:
\begin{enumerate}
\item $t_+s_-=\widehat{s}_-\widehat{t}_+p_s$.
\item $s_+t_-=p_s\widehat{t}_-\widehat{s}_+$.
\end{enumerate}
\end{lemma}
\begin{proof} $\mbox{ }$\vspace{.2truecm}

(1) If $\widehat{s}t=\widehat{t}s$, then $\widehat{s}_+t_+=\widehat{t}_+s_+$. Thus, $t_+=\widehat{s}_-\widehat{s}_+t_+=\widehat{s}_-\widehat{t}_+s_+$, and hence
$$t_+s_-=\widehat{s}_-\widehat{t}_+s_+s_-=\widehat{s}_-\widehat{t}_+p_s.$$

(2) If $\widehat{s}t=\widehat{t}s$, then $t_-\widehat{s}_-=s_-\widehat{t}_-$. Thus, $t_-=t_-\widehat{s}_-\widehat{s}_+=s_-\widehat{t}_-\widehat{s}_+$, and hence
$$s_+t_-=s_+s_-\widehat{t}_-\widehat{s}_+=p_s\widehat{t}_-\widehat{s}_+.$$
\end{proof}

Now, we will define the key object of this section.

\begin{definition}\label{newring2}
{\rm Under the above hypotheses, we define the set
$$S_-AS_+:=\{ s_-as_+ \mid s\in S, a\in A\}\subset \sopas .$$
}
\end{definition}

We proceed to fix the basic properties of this set. Notice that, for any $a\in A$ and any $s\in S$, we have $s_-as_+=s_-(p_sap_s)s_+$.

\begin{lemma}\label{Lemma3nou}
The set $S_-AS_+$ is a unital subring of $\sopas$, containing $A$ as unital subring.
\end{lemma}
\begin{proof}
Clearly, $1=s_- \cdot 1\cdot s_+\in S_-AS_+$. 

Now, let $s,t\in S$, $a,b\in A$. If $\widehat{s}t=\widehat{t}s$, then:
\begin{enumerate}
\item $t_-at_+\cdot s_-bs_+=t_-p_t ap_t t_+\cdot  s_-p_s b p_s s_+=t_- p_t ap_t \widehat{s}_- \widehat{t}_+ p_s b p_s s_+$
(the last equality is due to Lemma \ref{Lemma1nou}), and by Definition \ref{newring} it equals $(\widehat{s}t)_-\left(\alpha_{\widehat{s}}(p_t ap_t)\alpha_{\widehat{t}}(p_s b p_s)\right)(\widehat{s}t)_+$.
\item $(t_-at_+) + (s_-bs_+)=(t_-p_t ap_t t_+) +  (s_-p_s b p_s s_+)=$ $(t_-\widehat{s}_-\widehat{s}_+p_t ap_t \widehat{s}_-\widehat{s}_+t_+) +$\newline  $(s_-\widehat{t}_-\widehat{t}_+p_s b p_s \widehat{t}_-\widehat{t}_+s_+)$, and by Definition \ref{newring} it equals $(\widehat{s}t)_-\left(\alpha_{\widehat{s}}(p_t ap_t) +\alpha_{\widehat{t}}(p_s b p_s)\right)(\widehat{s}t)_+$.
\end{enumerate}
Hence, $S_-AS_+$ is a unital subring of $\sopas$. 

Finally, for any $a\in A$ and any $s\in S$ we have that $a=s_-s_+as_-s_+=s_-\alpha_s(a)s_+\in S_-AS_+$, so we are done.
\end{proof}

\begin{remark}\label{Remark4nou}
{\rm If $S$ acts on $A$ by corner isomorphisms, since $\alpha_s(A)=p_sAp_s$, then there exists a (unique) $b\in A$ such that $p_sap_s=\alpha_s(b)$. Hence, $s_-as_+=s_-(p_sap_s)s_+=s_-\alpha_s(b)s_+=s_-s_+bs_-s_+=b$. Thus, $S_-AS_+=A$.
}
\end{remark}

We have a picture of $S_-AS_+$ which simplifies the effective computation of this ring. Concretely

\begin{lemma}\label{Lemma:DirectLimit1}
The ring $S_-AS_+$ is isomorphic to a direct limit of rings.
\end{lemma}
\begin{proof}
For each $t\in S$, the set $t_-At_+$ is as unital subring of $\sopas$ containing $A$ as unital subring. Now, let $a,b\in A$ and $s,t\in S$. If $\widehat{s}t=\widehat{t}s$, then $t_-at_+=t_-\widehat{s}_-\widehat{s}_+p_t ap_t \widehat{s}_-\widehat{s}_+t_+= (\widehat{s}t)_-[\alpha_{\widehat{s}}(p_t ap_t)](\widehat{s}t)_+$, and similarly $s_-bs_+=(\widehat{s}t)_-[\alpha_{\widehat{t}}(p_s b p_s)](\widehat{s}t)_+$. Thus, $\left( s_-As_+\right)_{\{s\in S\}}$ is a direct system of rings, and clearly $S_-AS_+$ coincides with the direct union of this system. Moreover, for any $t\in S$ the rule $t_-at_+\mapsto p_tap_t$ defines a unital ring isomorphism from $t_-At_+$ to the corner ring $p_tAp_t$. Notice that, under this identification, the above inclusion map $t_-At_+\hookrightarrow (\widehat{s}t)_-A(\widehat{s}t)_+$ becomes $\alpha_{{\widehat{s}}\vert_{p_tAp_t}}:p_tAp_t\rightarrow p_{\widehat{s}t}Ap_{\widehat{s}t}$, so that
$$S_-AS_+\cong \varinjlim \left( p_tAp_t, \alpha_{{\widehat{s}}\vert_{p_tAp_t}}\right)$$
as unital rings.
\end{proof}

Next step is to show how the action $\alpha$ extends from $A$ to $S_-AS_+$.

\begin{lemma}\label{Lemma5nou}
The action $\alpha$ of $S$ on $A$ extends to an action $\overline{\alpha}:S\rightarrow \ndr(S_-AS_+)$ by corner isomorphisms.
\end{lemma}
\begin{proof}
For each $s\in S$, and for each $t_-at_+\in S_-AS_+$, we define $\overline{\alpha}_s(t_-at_+):=s_+t_-at_+s_-$. Let us see that it is well-defined. For, if $\widehat{s}t=\widehat{t}s$, then by Lemma \ref{Lemma1nou} $s_+t_-=\widehat{t}_-\widehat{s}_+p_t$ and $t_+s_-=p_t\widehat{s}_-\widehat{t}_+$. Thus, $s_+t_-at_+s_-=(s_+t_-)p_tap_t(t_+s_-)=\widehat{t}_-(\widehat{s}_+p_tap_t\widehat{s}_-)\widehat{t}_+=\widehat{t}_-\alpha_{\widehat{s}}(p_tap_t)\widehat{t}_+\in S_-AS_+$. Notice that $t_-at_+=0$ if and only if $p_tap_t=0$, so that there is no ambiguity related to the choice of $\widehat{s}$ and $\widehat{t}$. Moreover, if $t_-at_+\ne 0$, since $p_tap_t\ne 0$ and $\alpha_{\widehat{s}}$ is injective, we have 
$$0\ne \alpha_{\widehat{s}}(p_tap_t)=\widehat{s}_+p_tap_t\widehat{s}_-=\widehat{s}_+t_+t_-at_+t_-\widehat{s}_-=$$
and since $\widehat{s}t=\widehat{t}s$, it equals
$$=\widehat{t}_+(s_+t_-at_+s_-)\widehat{t}_-=\widehat{t}_+(\overline{\alpha}_s(t_-at_+))\widehat{t}_-.$$
Thus, $\overline{\alpha}_s(t_-at_+)\ne 0$, whence $\overline{\alpha}_s$ is injective.

Clearly, since $t_-t_+=1$ for every $t\in S$, $\overline{\alpha}_t$ is a homomorphism for any $t\in S$. A simple computation shows that for any $s,t\in S$ we have $\overline{\alpha}_{st}=\overline{\alpha}_s\cdot \overline{\alpha}_t$, whence $\overline{\alpha}$ defines an action of $S$ on $S_-AS_+$ by injective homomorphisms.

Finally, on one side, $\overline{\alpha}_t(s_-as_+)=t_+s_-as_+t_-=p_t(t_+s_-as_+t_-)p_t\in p_t(S_-AS_+)p_t$; on the other side, $p_ts_-as_+p_t=t_+(t_-s_-as_+t_+)t_-=t_+((st)_-a(st)_+)t_-=\overline{\alpha}_t((st)_-a(st)_+)\in \overline{\alpha}_t(S_-AS_+)$, as desired.
\end{proof}

Next result fixes the relation between $\sopas$ and $\SopaS$.

\begin{lemma}\label{aquestcalia}
$\sopas = \SopaS$.
\end{lemma}
\begin{proof}
Fix the inclusion map $\iota:S_-AS_+\hookrightarrow \sopas$ and the action $\overline{\alpha}$. Then, in $\sopas$ we have for any $t\in S$ and for any $s_-as_+\in S_-AS_+$:
\begin{enumerate}
\item $t_+(s_-as_+)=\overline{\alpha}_t(s_-as_+)t_+$.
\item $(s_-as_+)t_-=t_-\overline{\alpha}_t(s_-as_+)$.
\item $t_-t_+=1$.
\item $t_+t_-=\alpha_t(1)=\overline{\alpha}_t(1)=p_t$.
\end{enumerate}
Thus, by Definition \ref{newring} there exists a unique natural homomorphism 
$$\SopaS \rightarrow \sopas$$ 
induced by $\iota$ and $\overline{\alpha}$. Since $A\subset S_-AS_+\subset \sopas$ as unital rings and $\overline{\alpha}\vert_A=\alpha$, the result is clear.
\end{proof}

\begin{remark}\label{Remark_Laca1}{\rm
In terms of convariant representations Lemma \ref{aquestcalia} says that, given any algebraic dynamical system $(A,S,\alpha)$ with $S$ a  left Ore, left reversible monoid acting on $A$ by injective homomorphisms with not necessarily hereditary range, we can construct a new algebraic dynamical system $(S_-AS_+, S, \overline{\alpha})$ in which $S$ acts on $S_-AS_+$ by injective homomorphisms with hereditary range, and such that both dynamical systems have the same universal initial object.
}
\end{remark}

\begin{remark}\label{unitalmaps}
{\rm If $\alpha$ is an action by unital homomorphisms, then Lemma \ref{Lemma5nou} shows that $\overline{\alpha}_s$ is an automorphism of $S_-AS_+$ for every $s\in S$. Hence, $\overline{\alpha}:S\rightarrow \text{Aut}(S_-AS_+)$.
}
\end{remark}

\section{Dilations and full corners revisited}

Paschke \cite{Paschke}, generalizing previous results of Cuntz and R\o rdam, showed that a C*-algebra crossed product by an endomorphism corresponds naturally to a corner in a crossed product by an automorphism. In other words, $A\rtimes_{\alpha}\N$ is isomorphic
to a full corner $e(B\times_{\alpha'} \Z)e$ of a suitable group $C^*$-crossed product $B\times_{\alpha'} \Z$. Subsequently, Laca \cite{Laca} extended the scope of this result to semigroups $C^*$-crossed products on cancellative left Ore monoids. From an algebraic point of view, the analog result is \cite[Proposition 3.8]{AGGBP}. This result shows that, whenever the homomorphisms $\alpha_s$ are corner isomorphisms for all $s\in S$, then $\sop*_\alpha A*_\alpha S$ is a full corner ring $e(B*_{\alpha'}G)e$, where $B*_{\alpha'}G$ is a suitable skew group ring over the group $G=S^{-1}S$. In this section we will show that this result holds without the assumption that $S$ acts on $A$ by corner isomorphisms.\vspace{.2truecm} 

First, we will briefly recall the facts in \cite[Section 3]{AGGBP}.

\begin{noname}[{\cite[\textbf{3.1}]{AGGBP}}]\label{data3} {\rm
Let $A$ be a unital ring, $G$ a group, and $\alpha
:G\rightarrow \aut(A)$ an action by injective homomorphisms. Assume that $S$ is
a submonoid of $G$ with $G=S^{-1}S$ (thus, $S$ is cancellative and satisfies the left Ore
condition), and let $R=A\ast _{\alpha
}G$ be the corresponding skew group ring. Suppose that there exists a nontrivial idempotent $e\in A$ such that $\alpha _s(e)\leq e$ for all $s\in S$. Here, $\leq$ denotes the classical order for idempotents of a ring $A$: given $e,f\in A$ idempotents, $e\leq f$ if $e=ef=fe$. }
\end{noname}

\begin{remark}\label{norole}
{\rm
It is important to notice that, in the arguments given in \cite[Section 3]{AGGBP} for the results that follow, the fact that $\alpha$ acts by corner isomorphisms does not play any role, and thus these results hold under the hypothesis of $S$ acting on $A$ by injective homomorphisms.
}
\end{remark}

\begin{proposition}[{\cite[Lemma 3.2 \& Proposition 3.3]{AGGBP}}]\label{fullcorner1}
\label{isomap} Under the assumptions in (\ref{data3}), the following hold:
\begin{enumerate}
\item The action $\alpha$ restricts to an action $\alpha
':S\rightarrow \ndr(eAe)$ by injective
homomorphisms.
\item There are natural monoid homomorphisms $\sop \rightarrow
eRe$, given by $s\mapsto es^{-1}$, and
$S\rightarrow eRe$, given by $t\mapsto te$, satisfying the
conditions {\rm (1)--(4)} in Definition {\rm \ref{newring}} with
respect to
$\alpha '$ and the inclusion map $\phi :eAe\rightarrow
eRe$.
\item The rings
$\sopasprime$ and
$e(A\ast _{\alpha }G)e$ are isomorphic as $G$-graded rings.
\end{enumerate}
\end{proposition}

Now, we recall what happens in the reverse direction, looking for the representation
of a fractional skew monoid ring $\sopas$ as a corner ring of a suitable skew group ring.

\begin{noname}[{\cite[\textbf{3.5}]{AGGBP}}]\label{data4} {\rm
Suppose that $S$ is a cancellative left Ore monoid with enveloping group $G=S^{-1}S$ and $\alpha :S\rightarrow
\ndr(A)$ is an action of $S$ on $A$ by \underline{corner
isomorphisms}. Then, let us construct a ring $\sinva$ as in \cite{pic}. First, define a relation $\sim$ on $S\times A$ as follows: $(s_1,a_1)\sim (s_2,a_2)$ if and only if there exist $t_1,t_2\in
S$ such that $t_1s_1=t_2s_2$ and $\alpha_{t_1}(a_1)= \alpha_{t_2}(a_2)$. This is an equivalence relation \cite[Lemma 2.1]{pic}, and we
write
$[s,a]$ for the equivalence class of a pair $(s,a)$. Let $\sinva=
(S\times A)/{\sim}$ be the set of these equivalence classes. The left
Ore condition guarantees ``common denominators'' in $\sinva$. By
\cite[Lemma 2.2 ff.]{pic}, there are well-defined associative
multiplication and addition on $\sinva$. For, given any $[s_1,a_1], [s_2,a_2]\in \sinva$, choose $t_1,t_2\in
S$ such that $t_1s_1=t_2s_2$, and set: (i) $[s_1,a_1]\cdot [s_2,a_2]= [t_1s_1,\alpha_{t_1}(a_1) \alpha_{t_2}(a_2)]$; (ii) $[s_1,a_1]+ [s_2,a_2]= [t_1s_1,\alpha_{t_1}(a_1)+ \alpha_{t_2}(a_2)]$. The distributive law is also routine, and so $\sinva$
becomes a non-unital ring with a distinguished idempotent $[1_S, 1_A]$.
}
\end{noname}
\vspace{.1truecm}

\begin{remark}\label{Rem:Laca}
{\rm This procedure can be seen as a different way for obtaining Laca's construction of $A_S$ \cite{Laca}. For, we will proof that $\sinva$ is isomorphic to the direct limit algebra defined by Laca. Indeed, we consider the (upwards) direct system of rings $(A_s, f_{s, ts} )_{\{s,t\in S\}}$, where $A_s:=A$ for every $s\in S$, while $f_{s,ts}:A_s\rightarrow A_{ts}$ is defined by the rule $f_{s,ts}(a)=\alpha_t(a)$. If we denote $A_S:=\varinjlim (A_s, f_{s, ts} )$, it is clear that $a_1\in A_{s_1}$ and $a_2\in A_{s_2}$ will represent the same element in $A_S$ if and only if there exist $t_1,t_2\in S$ such that $t_1s_1=t_2s_2$ and $\alpha_{t_1}(a_1)=\alpha_{t_2}(a_2)$. Now, for each $s\in S$ we define a map $\varphi_s: A_s\rightarrow \sinva$ by the rule $\varphi_s (a)=[s,a]$. This is a well-defined ring morphism, and for any $t\in S$ it is easy to see that $\varphi_s=\varphi_{ts}\circ f_{s,ts}$. Thus, there exists a unique ring morphism $\Phi: A_S\rightarrow \sinva$ sendind $[a_s]$ to $[s,a_s]$. Clearly, $\Phi$ is onto. On the other side, if $0=\Phi([a_s])$, then $0=[s, a_s]$, so that there exists $t\in S$ such that $\alpha_t(a_s)=0$. Hence, $[a_s]=0$, whence $\Phi$ is one-to-one, and thus an isomorphism.
}
\end{remark}

Next, let extend $\alpha$ to an action of $S$ on $\sinva$. 

\begin{lemma}[{\cite[Lemmas 3.6 \& 3.7]{AGGBP}}]\label{actionsinva}$\mbox{ }$
\begin{enumerate}
\item The action of $\alpha$ on $A$ extends to an action $\widehat{\alpha}: S\rightarrow
\aut(\sinva)$. Concretely, given any $s\in S$ and $[t,a]\in \sinva$, set $\alsom_s([t,a])=
[s',\alpha_{t'}(a)]$ for $s',t'\in S$ such that $s's=t't$.
\item The rule $a\mapsto [1_S,a]$ defines an $S$-equivariant ring embedding
$\phi: A\rightarrow \sinva$ with image $[1_S,1_A]\cdot \sinva\cdot [1_S,1_A]$.
\end{enumerate}
\end{lemma}

\begin{remark}\label{nocancelnocorner}
{\rm $\mbox{ }$
\begin{enumerate}
\item Under Laca's picture of $\sinva$, given in Remark \ref{Rem:Laca}, the definition of $\widehat{\alpha}_s$ is exactly the one stated by Laca in \cite{Laca}.
\item The hypothesis that $\alpha_s(A)=p_sAp_s$ for every $s\in S$ is \underline{only necessary} to prove that $\phi (A)=[1_S,1_A]\cdot \sinva\cdot [1_S,1_A]$.
\item The hypothesis of $S$ being cancellative can be weakened to $S$ being left reversible, and the construction of the ring $S^{-1}A$ still works correctly (c.f. \cite[Lemmas 2.1 \& 2.2]{pic}). Moreover, the action of $\alpha$ on $A$ still extends to an action $\alpha: S\rightarrow \aut(\sinva)$ with the same definition, and the map $\phi$ is still a $S$-equivariant embedding (c.f. \cite[Theorem 2.4]{pic}).
\end{enumerate}
}
\end{remark}

Hence, we obtain the dilation result for fractional skew monoid rings in the case of actions given by corner isomorphisms.

\begin{proposition}[{\cite[Proposition 3.8]{AGGBP}}]\label{genunit}
Let $G$ be a group and $S$ a submonoid of $G$ such
that $G=S^{-1}S$. Let $\alpha :S\rightarrow \ndr(A)$ be
an action of $S$ on $A$ by corner isomorphisms. Then there exist a
an action $\alsom: G\rightarrow \aut(\sinva)$, and a nonzero
idempotent $e$ in $\sinva$ such that $\alsom _s(e)\le e$ for all $s\in
S $ and
$$\sopas \cong e((\sinva)*_{\hat\alpha}G)e$$ 
{\rm(}as $G$-graded
rings{\rm)}.
\end{proposition}

Certainly, Proposition \ref{genunit} is the converse of Proposition \ref{fullcorner1} under actions by corner isomorphisms. Now, applying these results and those of Section 2, we obtain the main result of the paper.

\begin{theorem}\label{maintheorem}
Let $A$ be a unital ring, let $S$ be a cancellative left Ore monoid with enveloping group $G=S^{-1}S$, and let $\alpha :S\rightarrow \ndr(A)$ be an action of $S$ on $A$. If $I:=\bigcup_{s'\in S}
\ker(\alpha_{s'})$, then $\alpha$ extends to an action $\widehat{\alpha}:G\rightarrow \aut(S^{-1}(S_-(A/I)S_+))$, and there exists a nonzero idempotent $e\in (S^{-1}(S_-(A/I)S_+))$ such that $\alsom _s(e)\le e$ for all $s\in
S $ and
$$\sopas \cong e((S^{-1}(S_-(A/I)S_+))*_{\widehat{\alpha}}G)e$$
as $G$-graded rings.
\end{theorem}
\begin{proof}
By Proposition \ref{kernel}(4), 
$$\sop*_\alpha A*_\alpha S= \sop*_{\alpha'} (A/I)*_{\alpha'} S$$ 
for an action $\alpha': S\rightarrow \ndr(A/I)$ by injective homomorphisms. Then, by Lemma \ref{aquestcalia}, 
$$ \sop*_{\alpha'} (A/I)*_{\alpha'} S = \sop*_{\overline{\alpha}} S_-(A/I)S_+*_{\overline{\alpha}} S,$$ where $\overline{\alpha}$ is an action by corner isomorphisms by Lemma \ref{Lemma5nou}. Thus, we can apply Proposition \ref{genunit} to $\sop*_{\overline{\alpha}} S_-(A/I)S_+*_{\overline{\alpha}} S$, so we are done.
\end{proof}

As an immediate consequence we have

\begin{corollary}\label{maintheorem2}
Let $A$ be a unital ring, let $S$ be a cancellative left Ore monoid with enveloping group $G=S^{-1}S$, and let $\alpha :S\rightarrow \ndr(A)$ be an action of $S$ on $A$ by injective homomorphisms. Then, $\alpha$ extends to an action $\widehat{\alpha}:G\rightarrow \aut(S^{-1}(S_-AS_+))$, and there exists a nonzero idempotent $e\in (S^{-1}(S_-AS_+))$ such that $\alsom _s(e)\le e$ for all $s\in
S $ and
$$\sopas \cong e((S^{-1}(S_-AS_+))*_{\widehat{\alpha}}G)e$$
as $G$-graded rings.
\end{corollary}

Because of Remark \ref{norole}, Corollary \ref{maintheorem2} is the converse of Proposition \ref{fullcorner1} for actions by injective homomorphisms. If $\alpha$ acts by \underline{unital} injective homomorphisms, then Remark \ref{unitalmaps} and Lemma \ref{Lemma5nou} imply that $\overline{\alpha}$ acts by unital automorphisms of $S_-AS_+$. Hence, 
$$\SopaS = (S_-AS_+)*_{\overline{\alpha}}G$$ 
by Definition \ref{newring}, and thus we have

\begin{corollary}\label{unitalmaps2}
Let $A$ be a unital ring, let $S$ be a cancellative left Ore monoid with enveloping group $G=S^{-1}S$, and let $\alpha :S\rightarrow \ndr(A)$ be an action of $S$ on $A$ by injective unital homomorphisms. Then, 
$$\sopas = (S_-AS_+)*_{\overline{\alpha}}G.$$
\end{corollary}

\begin{remark}\label{Remark_Laca2bis}{\rm
In terms of covariant representations, Corollary \ref{maintheorem2} says that, because of Remark \ref{Remark_Laca1}, we can replace any algebraic dynamical system $(A,S,\alpha)$ in which $\alpha$ does not act by corner isomorphisms (``with not necessarily hereditary range'' in $C^*$-algebra terms) by a new one $(S_-AS_+, S, \overline{\alpha})$ in which $\overline{\alpha}$ acts by corner isomorphisms (``with hereditary range'' in $C^*$-algebra terms), so that the algebraic dilation construction applies. Moreover, both dynamical systems share the same universal initial object. So, Corollary \ref{unitalmaps2} means that the construction in \cite[Section 3]{AGGBP} give us a dilation result even in the case of unital homomorphisms.
}
\end{remark}

Let us close this section by giving an example of application of Corollary \ref{unitalmaps2}, which benefits from Lemma \ref{Lemma:DirectLimit1} for the computation of $S_-AS_+$.

\begin{example}\label{Example1}
{\rm For any natural number $n\geq 2$ and any field $K$, consider $L_n$ the universal $K$-algebra generated by elements $x_1, \dots, x_n, y_1, \dots , y_n$ satisfying the relations: (i) $x_iy_j=\delta_{i,j}$ for every $1\leq i,j\leq n$; (ii) $1=\sum\limits_{i=1}^ny_ix_i$; these algebras are known as Leavitt algebras \cite{Lvtt}, and are the algebraic counterpart of Cuntz algebras \cite{C1}. Let $\alpha\in \mbox{End}(L_n)$ the unital endomorphism defined by the rule $\alpha (a)=\sum\limits_{i=1}^ny_iax_i$, which is outer. Now, consider the fractional skew monoid ring 
$${{\Z^+}^{\text{\rm op}}}*_{\alpha} L_n*_{\alpha} \Z^+.$$ 
According to Corollary \ref{unitalmaps2},
$${{\Z^+}^{\text{\rm op}}}*_{\alpha} L_n*_{\alpha} \Z^+ = (S_-L_nS_+)*_{\widehat{\alpha}}\Z$$
as $\Z$-graded algebras. By Lemma \ref{Lemma:DirectLimit1},
$$S_-L_nS_+\cong \varinjlim \left( L_n, \alpha^m \right).$$
Now, in view of the fact that the set $\{ y_ix_j\}_{1\leq i,j\leq n}$ is a system of matrix-units for the isomorphisms $L_n\cong M_n(L_n)$, it is easy to see that $\alpha:L_n\rightarrow L_n$ acts as the diagonal embedding from $L_n$ to $M_n(L_n)$. Hence, 
$$S_-L_nS_+\cong \varinjlim \left( L_n, \alpha \right)\cong M_{n^{\infty}}(L_n),$$
so that 
$$(S_-L_nS_+)*_{\widehat{\alpha}}\Z\cong M_{n^{\infty}}(L_n)*_{\overline{\alpha}}\Z$$
for an outer action $\overline{\alpha}$ of $\Z$ on $M_{n^{\infty}}(L_n)$. Clearly, $M_{n^{\infty}}(L_n)$ is a purely infinite simple ring (a property for rings analog to purely infinite simple $C^*$-algebras, see \cite{AGP} for a formal definition). Hence, an easy adaptation of the results in \cite[Section 4]{AGGBP}  (see e.g. \cite[Theorem 1.2]{LP}) shows that $M_{n^{\infty}}(L_n)*_{\overline{\alpha}}\Z$, and thus ${{\Z^+}^{\text{\rm op}}}*_{\alpha} L_n*_{\alpha} \Z^*$, is a purely infinite simple ring.
}
\end{example}

\section{Semigroup $C^*$-crossed products}

In this section, we consider the application of the results in the previous sections to the case of unital $C^*$-algebras. \vspace{.2truecm}

First notice that, given a unital $C^*$-algebra $A$ and a left Ore, left reversible monoid $S$ acting via $\alpha$ by injective $\ast$-endomorphisms of $A$, the associated dynamical system $(A,S, \alpha)$ satisfies the requirements of Laca's construction of the semigroup crossed product $A\times_{\alpha} S$. 

\begin{lemma}\label{KeyLink}
Let $A$ be a unital $C^*$-algebra, let $S$ be a left Ore, left reversible monoid, and let $\alpha$ be an action of $S$ by not necessarily unital injective $\ast$-endomorphisms of $A$. Then, $\sopas$ is a dense $\ast$-subalgebra of $A\times_{\alpha}S$.
\end{lemma}
\begin{proof}
Let $(\phi_A, \phi_S)$ be the universal covariant pair of $\sopas$, and let $(i_A, i_S)$ be the universal covariant pair of $A\times_{\alpha}S$. Notice that both are associated to the same representation of $A\times_{\alpha}S$ on a Hilbert space $H$. By the universal property of $\sopas$, applied to the covariant pair $(i_A, i_S)$, there exists a unique $\ast$-homomorphism
$$\varphi: \sopas \rightarrow A\times_{\alpha}S$$
such that $i_A=\varphi \circ \phi_A$ and $i_S=\varphi \circ \phi_S$. Moreover,  (2) in Definition \ref{newalgebra} applies, by the above remark, to the covariant pair $(\phi_A, \phi_S)$. So, there exists a representation $\phi_A\times \phi_S$ such that $\phi_A=(\phi_A\times \phi_S)\circ i_A$ and $\phi_S=(\phi_A\times \phi_S)\circ i_S$.

Now, by Proposition \ref{kernel}, $\phi_A$ is injective, and then so it is $i_A$ by the above argument. By the same argument $\varphi$ restricts to a $\ast$-isomorphism of monoids between $\phi_S(S)$ and $i_S(S)$. Hence, $\varphi$ is injective.

Finally, by (3) in Definition \ref{newalgebra}, $\text{im}(\varphi)$ is a dense $\ast$-subalgebra of $A\times_{\alpha}S$, so we are done.
\end{proof}

Hence, we obtain a slight improvement of Laca's result. Concretely, we do not require $S$ to be cancellative in order to realize $A\times_{\alpha} S$ as a full corner of a group $C^*$-crossed product, and thus guarantee that it is nontrivial. In particular, notice that this approach skips the (implicit) requirement of Laca, which asks the endomorphisms to have hereditary range, a technical fact used to construct the dilation.\vspace{.2truecm}

By Lemma \ref{KeyLink}, $A\times_{\alpha} S$ is nontrivial, contains a $\ast$-isomorphic copy of $\sopas$, and moreover it is the norm completion of $\sopas$ in a suitable norm. Under this picture, notice that $A\subset S_-AS_+\subset \sopas \subset A\times_{\alpha} S$ as unital algebras. So, we can define $S^*AS$ to be the norm-completion of $S_-AS_+$ under the norm inherited by the inclusion. Thus, we can transfer all the results obtained in Section 2 to the context of $C^*$-algebras. Concretely we have

\begin{lemma}
$S^*AS$ is a unital sub-$C^*$-algebra of $A\times_{\alpha} S$, containing $A$ as unital sub-$C^*$-algebra.
\end{lemma}

Certainly Lemma \ref{Lemma:DirectLimit1}, which allows to present $S_-AS_+$ as $\varinjlim \left( p_tAp_t, \alpha_{{\widehat{s}}\vert_{p_tAp_t}}\right)$, also applies when we consider the direct limit construction in the category of $C^*$-algebras and $\ast$-homomorphisms. Thus, we have

\begin{lemma}\label{Lemma:DirectLimit1SeaStar}
The $C^*$-algebra $S^*AS$ is isomorphic to a direct limit of $C^*$-algebras.
\end{lemma}

As a consequence, in the same manner as in Section 2, we can prove 

\begin{proposition}\label{Prop:Paschke1}
$\mbox{ }$
\begin{enumerate}
\item The action $\alpha$ of $S$ on $A$ extends to an action $\overline{\alpha}:S\rightarrow \ndr(S^*AS)$ by corner $\ast$-isomorphisms (so that $\overline{\alpha}_s$ has hereditary range for every $s\in S$).
\item $A\times_{\alpha} S=(S^*AS)\times_{\overline{\alpha}} S$.
\end{enumerate}
\end{proposition}

\begin{remark}\label{Remark_Laca2}{\rm
In terms of covariant representations, Proposition \ref{Prop:Paschke1} says that we can replace $(A,S, \alpha)$ by a new dynamical system $(S^*AS, S, \overline{\alpha})$ such that the maps $\overline{\alpha}_s$ have hereditary range for every $s\in S$, while both dynamical systems have the same universal initial object associated. Thus, we are extending the scope of Laca's arguments to actions in which having hereditary range is not required.}
\end{remark}

Now, if $S$ is cancellative, since $S^{-1}A$ can be seen as a direct limit of $C^*$-algebras \cite{Laca} (see Remark \ref{Rem:Laca}), we can assume that $S^{-1}A$ denotes the corresponding $C^*$-algebra, whence the related results in \cite[Section 3]{AGGBP} apply for unital $C^*$-algebras. Hence, we have

\begin{theorem}\label{forseastar}
Let $A$ be a unital $C^*$-algebra, let $S$ be a cancellative left Ore monoid with enveloping group $G=S^{-1}S$, and let $\alpha :S\rightarrow \ndr(A)$ be an action of $S$ on $A$ by injective $\ast$-homomorphisms. Then, $\alpha$ extends to an action $\widehat{\alpha}:G\rightarrow \aut(S^{-1}(S^*AS))$, and there exists a full projection $e\in (S^{-1}(S^*AS))$ such that $\alsom _s(e)\le e$ for all $s\in
S $ and
$$A\times_{\alpha} S \cong e((S^{-1}(S^*AS))\times_{\widehat{\alpha}}G)e$$
is an $S$-equivariant $\ast$-isomorphism.
\end{theorem}

Theorem \ref{forseastar} means that, because of Proposition \ref{Prop:Paschke1}, we can extend the scope of Laca's techniques to the case of an action whose range is not necessarily hereditary. The extreme case of this situation occurs when all the maps $\alpha_s$ are unital and injective, but not isomorphisms. This is the situation of Cuntz's original question, that we answer in the affirmative.

\begin{corollary}\label{unitalinjseastar}
Let $A$ be a unital $C^*$-algebra, let $S$ be a cancellative left Ore monoid with enveloping group $G=S^{-1}S$, and let $\alpha :S\rightarrow \ndr(A)$ be an action of $S$ on $A$ by injective unital $\ast$-homomorphisms. Then, 
$$A\times_{\alpha} S = (S^*AS)\times_{\overline{\alpha}}G$$
is an $S$-equivariant $\ast$-isomorphism.
\end{corollary}

\begin{example}\label{Example2}
{\rm In analogy with Example \ref{Example1}, for any natural number $n\geq 2$ consider the $n$-th Cuntz algebra $\mathcal{O}_n$ , i.e. the $C^*$-algebra generated by pairwise orthogonal isometries $s_1, \dots, s_n$ satisfying that $1=\sum\limits_{i=1}^ns_is_i^*$. Let $\alpha\in \mbox{End}(\mathcal{O}_n)$ the unital endomorphisms defined by the rule $\alpha (a)=\sum\limits_{i=1}^ns_ias_i^*$, which is outer. Now, consider the (Paschke) crossed product 
$$\mathcal{O}_n \rtimes_{\alpha} \N.$$ 
According to Proposition \ref{Prop:Paschke1}(2),
$$\mathcal{O}_n \rtimes_{\alpha} \N = (S^*\mathcal{O}_nS)\times_{\widehat{\alpha}}\Z.$$
By Lemma \ref{Lemma:DirectLimit1SeaStar},
$$S^*\mathcal{O}_nS\cong \varinjlim \left( \mathcal{O}_n, \alpha^m \right).$$
Now, in view of the fact that the set $\{ s_is_j^*\}_{1\leq i,j\leq n}$ is a system of matrix-units for the isomorphisms $\mathcal{O}_n\cong M_n(\mathcal{O}_n)$, it is easy to see that $\alpha:\mathcal{O}_n\rightarrow \mathcal{O}_n$ acts as the diagonal embedding from $\mathcal{O}_n$ to $M_n(\mathcal{O}_n)$. Hence, 
$$S^*\mathcal{O}_nS\cong \varinjlim \left( \mathcal{O}_n, \alpha \right)\cong M_{n^{\infty}}(\mathcal{O}_n),$$
so that 
$$(S^*\mathcal{O}_nS)\times_{\widehat{\alpha}}\Z\cong M_{n^{\infty}}(\mathcal{O}_n)\times_{\overline{\alpha}}\Z$$
for an outer action $\overline{\alpha}$ of $\Z$ on $M_{n^{\infty}}(\mathcal{O}_n)$. Clearly, $M_{n^{\infty}}(\mathcal{O}_n)$ is a purely infinite simple $C^*$-algebra, and then so is $M_{n^{\infty}}(\mathcal{O}_n)\times_{\overline{\alpha}}\Z$ (whence $\mathcal{O}_n \rtimes_{\alpha} \N$) by \cite{JKO}).

Thus, we can compute $K$-Theory of $\mathcal{O}_n \rtimes_{\alpha} \N$ by using the Pimsner-Voiculescu exact sequence \cite{Black}:
$$\xymatrix{K_0(M_{n^{\infty}}(\mathcal{O}_n))\ar[r]^{{id-\overline{\alpha}}^*} &K_0(M_{n^{\infty}}(\mathcal{O}_n))\ar[r] & K_0(M_{n^{\infty}}(\mathcal{O}_n)\times_{\overline{\alpha}}\Z)\ar[d]\\
K_1(M_{n^{\infty}}(\mathcal{O}_n)\times_{\overline{\alpha}}\Z)\ar[u] & K_1(M_{n^{\infty}}(\mathcal{O}_n))\ar[l] & K_1(M_{n^{\infty}}(\mathcal{O}_n))\ar[l]_{{id-\overline{\alpha}}^*}
}
\,.$$
It is well-known that $K_1(M_{n^{\infty}}(\mathcal{O}_n))=0$, and it is easy to see that $K_0(M_{n^{\infty}}(\mathcal{O}_n))\cong \Z[\frac{1}{n}]$. Under this picture, ${id-\overline{\alpha}}^*$ is given by multiplication by $1-n$, so that it is injective. Hence, $K_1(\mathcal{O}_n \rtimes_{\alpha} \N)=0$, and thus $K_0(\mathcal{O}_n \rtimes_{\alpha} \N)\cong \Z[\frac{1}{n}]/(1-n)\Z[\frac{1}{n}]$. By Kirchberg-Phillips Theorem \cite{Kirch, Phil} we conclude that $\mathcal{O}_n \rtimes_{\alpha} \N\cong \mathcal{O}_n$.
}
\end{example}

\begin{remark}
{\rm Very recently Cuntz and Li have developed a theory for $C^*$-algebras associated to integral domains \cite{CL1}. As a consequence, a large amount of work on semigroup $C^*$-algebras has been done (see e.g \cite{CEL, CV, L}), specially when the action is given by \underline{unital} $\ast$-homomorphisms. From this point of view Corollary \ref{unitalinjseastar}, jointly with Lemma \ref{Lemma:DirectLimit1SeaStar}, could be a useful instrument to work on this line.}
\end{remark}

\section{Noncancellative left denominator monoids}

In this section we will briefly analyze what kind of result, analog to Theorem \ref{maintheorem}, can we expect when we weaken the hypotheses on $S$ from cancellative to left reversible. The reason for considering this situation relies on a construction, due to Exel, of crossed products of $C^*$-algebras by partial actions of groups (see e.g. \cite{Exel1}, and \cite{DE} for a purely algebraic analog). Concretely, Exel shows that we can associate, to a $C^*$-algebra $A$ and a partial action $\alpha$ of a (discrete) group $G$ on $A$, a $C^*$-partial crossed product algebra $A\times _{\alpha} G$; the more clear example of this construction is the Exel-Laca picture of Cuntz-Krieger algebras extended to infinite matrices \cite{ExelLaca}. Exel  \cite{Exel2} showed that there exist an inverse semigroup $\mathcal{S}(G)$ and a (global) action $\mathcal{S}(\alpha)$ of $\mathcal{S}(G)$ on $A$ such that the $C^*$-partial crossed product $A\times_{\alpha} G$ turns out to be isomorphic to the semigroup $C^*$-crossed product $A\times_{\mathcal{S}(\alpha)} \mathcal{S}(G)$ (see also \cite{BE}). Since all these inverse semigroups $\mathcal{S}(G)$ enjoys the original standing hypotheses (\ref{data1}) \cite[Example 1.5(2)]{pic}, the analysis of the dilation construction in this context could be a useful tool for studying Exel crossed products by partial actions of groups from the ``classical'' context of $C^*$-crossed products of groups.\vspace{.2truecm}

Let us fix then the concrete data for this section.

\begin{noname}\label{datafinal} 
{\rm
Let $A$ be a unital ring, let $S$ be a not necessarily cancellative left Ore, left reversible monoid, and $\alpha
:S\rightarrow \ndr(A)$ an action.
}
\end{noname}

Under these hypotheses, if $I:=\bigcup_{s'\in S}\ker(\alpha_{s'})$, then $\sop*_\alpha A*_\alpha S= \sop*_{\alpha'} (A/I)*_{\alpha'} S$ for an action $\alpha': S\rightarrow \ndr(A/I)$ by injective homomorphisms by Proposition \ref{kernel}. So, we can assume that the action is given by injective homomorphisms. 

Hence, the results in Section 2 apply, so that $\alpha$ extends to an action $\widehat{\alpha}:S\rightarrow \ndr(S_-AS_+)$ by corner isomorphisms and $\sopas=\SopaS$. Thus, we can assume that the action is given by corner isomorphisms. So, we can recast (\ref{datafinal}) as 

\begin{noname}\label{datafinal2} 
{\rm
Let $A$ be a unital ring, let $S$ be a not necessarily cancellative left Ore, left reversible monoid, and $\alpha
:S\rightarrow \ndr(A)$ an action by corner isomorphisms.
}
\end{noname}

As noticed in Remark \ref{nocancelnocorner}(2), under our hypotheses the construction of the ring $S^{-1}A$ still works correctly, the action of $\alpha$ on $A$ still extends to an action $\widehat{\alpha}: S\rightarrow \aut(\sinva)$ with the same definition, and the map $\phi$ is still a $S$-equivariant embedding.\vspace{.2truecm}

At this point, the only remaining question is the exact relation of the monoid $S$ with the possible associated groups which allows us to represent $\sopas$ as a sort of corner ring over an skew group ring $A*_{\beta}G$. There are  two concrete group constructions associated to $S$:
\begin{enumerate}
\item The monoid localization $G=S^{-1}S$. This is a group satisfying a universal property with respect to the natural map $\lambda: S\rightarrow G$, which is injective if and only if $S$ is cancellative \cite[Corollary 8.5]{Cohn}.
\item The monoid $\widetilde{S}:=\widehat{\alpha}(S)\leq \aut(\sinva)$. This is a cancellative left Ore monoid, so that by part (1) it embeds in a group $\widetilde{G}\leq \aut(\sinva)$ \cite[Proposition 2.6]{pic}.
\end{enumerate}

By the universal property of $G$, there exists a unique group morphism $\varphi: G\rightarrow \widetilde{G}$ such that $\widehat{\alpha}=\varphi \lambda$. The map $\varphi$ is one-to-one by the universal property of $G$, while it is onto by the universal property of $\widetilde{G}$. So, $\varphi$ is an isomorphism and moreover, $\widetilde{S}$ is isomorphic to $\lambda(S)$ through this isomorphism. In particular, for any $s,t\in S$, we have $\lambda (s)=\lambda (t)$ if and only if $\widehat{\alpha}_s=\widehat{\alpha}_t$. By Proposition \ref{genunit}, there exists an idempotent $e:=[1_{\widetilde{S}},1_A]\in {\widetilde{S}}^{-1}A$ such that
$${{\widetilde{S}}^{\text{\rm op}}}*_{\widehat{\alpha}} A*_{\widehat{\alpha}} {\widetilde{S}}=e({\widetilde{S}}^{-1}A*_{\widehat{\alpha}} {G})e$$
as $G$-graded rings.

Notice that $\phi:A\rightarrow {\widetilde{S}}^{-1}A$ remains an injective $\widetilde{S}$-equivariant homomorphisms, and its image coincide with $[1_{\widetilde{S}},1_A]({\widetilde{S}}^{-1}A)[1_{\widetilde{S}},1_A]$. Now, fixing the monoid homomorphisms $\sop \rightarrow {{\widetilde{S}}^{\text{\rm op}}}$ (given by the rule $t_-\mapsto \widetilde{t}_-$), $S\rightarrow \widetilde{S}$ (given by the rule $t_+\mapsto \widetilde{t}_+$), and the identity map $\text{id}: A\rightarrow A$, we can use the universal property of $\sopas$ to induce an onto ring homomorphism
$$\widehat{\lambda}:\sopas \twoheadrightarrow {{\widetilde{S}}^{\text{\rm op}}}*_{\widehat{\alpha}} A*_{\widehat{\alpha}} {\widetilde{S}}$$
Then, we conclude the following result, which generalizes Theorem \ref{maintheorem}.

\begin{theorem}\label{finaltheorem}
Let $A$ be a unital ring, let $S$ be a not necessarily cancellative left Ore, left reversible  monoid with enveloping group $G=S^{-1}S$, let $\lambda:S\rightarrow G$ be the natural map, let ${\widetilde{S}}=\lambda(S)$, and let $\alpha :S\rightarrow \ndr(A)$ be an action of $S$ on $A$. If $I:=\bigcup_{s'\in S}
\ker(\alpha_{s'})$, then $\alpha$ extends to an action $\widehat{\alpha}:G\rightarrow \aut({\widetilde{S}}^{-1}({\widetilde{S}}_-(A/I){\widetilde{S}}_+))$, and there exists a nonzero idempotent $e\in ({\widetilde{S}}^{-1}({\widetilde{S}}_-(A/I){\widetilde{S}}_+))$ such that $\alsom _s(e)\le e$ for all $s\in
S $ and
$$\Phi: \sopas \twoheadrightarrow e(({\widetilde{S}}^{-1}({\widetilde{S}}_-(A/I){\widetilde{S}}_+))*_{\widehat{\alpha}}G)e$$
is a $S$-equivariant onto ring homomorphism.
\end{theorem}

Certainly, the technology involved allows to transfer the results to the context of $C^*$-algebras with no additional effort, at least when the action is given by injective $\ast$-ho\-mo\-mor\-phisms. So, we have the following generalization of Theorem \ref{forseastar}.

\begin{theorem}\label{finaltheoremSeaStar}
Let $A$ be a unital $C^*$-algebra, let $S$ be a not necessarily cancellative left Ore, left reversible monoid with enveloping group $G=S^{-1}S$, let $\lambda:S\rightarrow G$ be the natural map, let ${\widetilde{S}}=\lambda(S)$, and let $\alpha :S\rightarrow \ndr(A)$ be an action of $S$ on $A$ by injective $\ast$-homomorphisms. Then, $\alpha$ extends to an action $\widehat{\alpha}:G\rightarrow \aut({\widetilde{S}}^{-1}({\widetilde{S}}^*A{\widetilde{S}}))$, and there exists a nonzero full projection $e\in {\widetilde{S}}^{-1}({\widetilde{S}}^*A{\widetilde{S}})$ such that $\alsom _s(e)\le e$ for all $s\in
S $ and
$$\Phi: A\rtimes_{\alpha}S \twoheadrightarrow e(({\widetilde{S}}^{-1}({\widetilde{S}}^*A{\widetilde{S}}))\times_{\widehat{\alpha}}G)e$$
is an $S$-equivariant onto $\ast$-homomorphism.
\end{theorem}

Unfortunately, it seems quite clear that $\mbox{Ker}(\Phi)$ is not a $\alpha$-invariant ideal. So, up to very particular cases we cannot expect to represent $\sopas$ (respectively $A\rtimes_{\alpha}S$) \underline{exactly} as a full corner of a suitable skew group ring (respectively a group $C^*$-crossed product). 

\section*{Acknowledgments}

This work was inspired on conversations with Joachim Cuntz rising during the visit of the author to the Mathematisches Forschungsinstitut Oberwolfach (Germany) to participate in the mini-workshop ``Endomorphisms, Semigroups and C*-algebras of Rings'' in April 2012. The author thanks Joachim Cuntz for his kindness, and the host center for its warm hospitality. Also, the author thanks the referee, whose suggestions deeply improved this paper.

\end{document}